\documentclass[reqno]{amsart}
\usepackage{float}
\usepackage{amsmath}
\usepackage{graphicx}
\usepackage{latexsym}
\usepackage{amsfonts}
\usepackage{amssymb}
\theoremstyle{plain}
\newtheorem{theorem}{Theorem}

\newtheorem{lemma}[theorem]{Lemma}
\newtheorem{proposition}[theorem]{Proposition}
\theoremstyle{definition}

\newtheorem*{remark*}{Remark}

\begin{document}
\title[Ordered random walks with heavy tails]{Ordered random walks with heavy tails}
\author[Denisov]{Denis Denisov}

\author[Wachtel]{Vitali Wachtel}

\thanks{Supported by the DFG}
\maketitle

\begin{center}
\begin{minipage}{6cm}
\begin{center}
{School of Mathematics\\
Cardiff University}\\
Senghennydd Road, CF24 4AG\\
Cardiff, Wales, UK\\
{DenisovD@cf.ac.uk}
\end{center}
\end{minipage}
\begin{minipage}{6cm}
\begin{center}
{Mathematical Institute\\
 University of Munich,\\
 Theresienstrasse 39, D--80333\\
Munich, Germany}\\
{wachtel@mathematik.uni-muenchen.de}
\end{center}
\end{minipage}
\end{center}
\vspace{1cm}



{\small
{\bf Abstract.}
This note continues paper of Denisov and Wachtel (2010), where  
we have constructed a $k$-dimensional random walk conditioned to stay in the Weyl chamber of type $A$. 
The  construction was done  under the assumption that the original random walk has $k-1$ moments. 
In this note we continue the study of killed random walks in the Weyl chamber, 
and assume that the tail of increments is regularly varying of index $\alpha<k-1$. 
It appears that the asymptotic behaviour of random walks is different in this case. 
We determine the asymptotic behaviour of the exit time, and, using this
information, construct a conditioned process which lives on a partial compactification of the Weyl chamber.
\\[1cm]
{\bf Key words: } Dyson's Brownian Motion, Doob $h$-transform, superharmonic function, Weyl chamber,
Martin boundary\\
{\bf AMS Subject Classification: } Primary 60G50; Secondary 60G40, 60F17\\
}
\section{Main results and discussion}
\subsection{Introduction}
This note is a continuation of our paper \cite{DW10}.  
In \cite{DW10} we constructed a $k$-dimensional random walk conditioned 
to stay in the Weyl chamber of type $A$.  
The condional version of the random walk was defined via Doob's $h$-transform. 
The form of the corresponding harmonic function has been suggested by Eichelsbacher and K\"onig \cite{EK08}. 
This construction was performed under the optimal moment conditions and required the existence of 
$k-1$ moments of the random walk.

The main aim of the present work is to consider the case, when that moment condition is not fulfilled.  
Instead of the existence of $(k-1)$-th moment of the increment, 
we shall assume that the tail function is regularly varying of index $2<\alpha<k-1$. 
This assumption significantly changes the behaviour of the random walk. It turns 
out that the asymptotic behaviour of the exit time from the Weyl chamber depends not only on the number of walks
but also on the index $\alpha$. The typical sample path behaviour for the occurence of large exit times 
is different as well.  The main reason for that is that the large exit times are caused by one (or several) 
big jumps of the random walk. 

We now introduce some notation. 
Let $S=(S_1,S_2,\ldots,S_k)$ be a $k$-dimensional random walk with
$$
S_j(n)=\sum_{i=1}^n X_j(i),
$$ 
where $\{X_j(i)\}_{i,j\geq1}$ are independent copies of a random variable $X$. Let $W$ denote the Weyl chamber of 
type $A$, i.e.,
$$
W=\left\{x\in\mathbb{R}^k:\, x_1<x_2<\ldots<x_k\right\}.
$$
Let $\tau_x$ denote the first exit time of random walk with starting point $x\in W$, that is,
$$
\tau_x:=\min\{n\geq1:\,x+S(n)\notin W\}.
$$
The main purpose of the present paper is to study the asymptotic behaviour of $\mathbf{P}(\tau_x>n)$
and  to construct a model for ordered random walks. 
Recall that in order to define a  random walk conditioned to stay in $W$,  
one should find a Doob $h$-transform 
$$
\mathbf E[h(x+S(1)),\tau_x>1]=h(x)>0,x\in W. 
$$
We say that the function which satisfies the latter condition is harmonic. 
However,  it seems that it is not possible to find a harmonic function for the Doob $h$-transform under present conditions. 
Therefore, we use a partial compactification of $W$, which is based on the of sample path behaviour 
of the random walk  $S$ on the event $\{\tau_x>n\}$. (Recall that a more formal way consists in applying an
$h$-transform with a harmonic function.) Finally, we prove a functional limit theorem for 
random walks conditioned to stay in the Weyl chamber up to big, but finite, time.

To simplify our proofs we shall restrict our attention 
to the case $\alpha\in(k-2,k-1)$. However, it will be clear from the proof, that our method works also for 
smaller values of $\alpha$.

\subsection{Tail distribution of $\tau_x$}
We shall assume that $\mathbf{E}X=0$. This assumption does not restrict the generality, since $\tau_x$ 
depends only on differences of coordinates of the random walk $S$. We consider a situation when 
increments have $k-2$ finite moments, i.e., $\mathbf{E}|X|^{k-2}<\infty$. Under this condition, for
$(S_1,S_2,\ldots,S_{k-1})$ we can construct a harmonic function by using results of \cite{DW10}. 
Denote this function by $V^{(k-1)}(x)$. It is easy to see that this function is 
superharmonic for our original $k$-dimensional random walk, i.e.,
$$
\mathbf{E}\left[V^{(k-1)}(x+S(1)),\tau_x>1\right]\leq V^{(k-1)}(x)
$$
and the inequality is strict at least for one $x\in W$.
Denote
\begin{align*}
v(x)&:=pv_1(x)+qv_2(x)\\
&:=pV^{(k-1)}(x_1,x_2,\ldots,x_{k-1})+qV^{(k-1)}(x_2,x_3,\ldots,x_{k}).
\end{align*}
This function is also superharmonic for all $p,q\geq0$.

To state our first result we introduce a convolution of $v$ with the Green function of random walk
in the Weyl chamber:
$$
U(x):=\sum_{l=0}^{\infty}\mathbf{E}\left[v(x+S(l)),\tau_x>l\right],\quad x\in W.
$$

\begin{theorem}\label{T1}
Assume that 
\begin{equation}\label{RegVar}
\mathbf{P}(X>x)\sim \frac{p}{x^{\alpha}}\text{ and }\mathbf{P}(X<-x)\sim \frac{q}{x^{\alpha}}\text{, as }x\to\infty,
\end{equation}
for some $\alpha\in(k-2,k-1)$ and some $k\geq4$. Then $U(x)$ is a strictly positive superharmonic function, i.e.,
$\mathbf{E}\left[U(x+S(1)),\tau_x>1\right]<U(x)$ for all $x\in W$. Moreover,
\begin{equation}
\label{T1.1}
\mathbf{P}(\tau_x>n)\sim\theta U(x)n^{-\alpha/2-(k-1)(k-2)/4}\quad\text{as }n\to\infty,
\end{equation}
where $\theta$ is an absolute constant.
\end{theorem}
There is a very simple strategy behind formula (\ref{T1.1}). 
For the event  $\{\tau_x>n\}$ to occur 
either the random walk
on the top or the random walk on the bottom should jump away, i.e., $X_k(l)\approx\sqrt{n}$ or
$X_1(l)\approx-\sqrt{n}$ for some $l\geq1$. After such a big jump we have a system of $k-1$ random walks with bounded
distances between each other and one random walk on the characteristic distance $\sqrt{n}$. 
This implies that the probability
that all $k$ random walks stay in W up to time $n$ is of the same order as the probability that $k-1$ random walk stay in 
$W$ up to time $n-l$. But it follows from (\ref{RegVar}) that $\mathbf{E}[|X|^{k-2}]<\infty$. So we can apply Theorem 1 
from \cite{DW10}, which says that the latter probability is of order $n^{-(k-1)(k-2)/4}$. 
Since $\mathbf{P}(|X|>\sqrt{n})\sim n^{-\alpha/2}$, we see that $\mathbf{P}(\tau_x>n)$ is of order $n^{-\alpha/2-(k-1)(k-2)/4}$.
This strategy sheds also some light on the structure of the function $U(x)$: 
the $l$-th summand in the series corresponds to the
case, when big jump occurs at time $l+1$.
\subsection{Construction of a conditioned random walk}
Since $U$ is not harmonic, we can not use the Doob $h$-transform with this function to define a random walk, 
conditioned to stay in $W$ for all times. (More precisely, an $h$-transform with a superharmonic function 
leads to strict substohastic transition kernel.) 
An alternative approach via distributional limit does not work as well: using 
Theorem \ref{T1} we can define $\hat{P}(x,A)$ for any $x\in W$ and for any bounded $A\subset W$ by the relation
\begin{align*}
\hat{P}(x,A)&=\lim_{n\to\infty}\mathbf{P}(x+S(1)\in A|\tau_x>n)\\
&=\lim_{n\to\infty}\int_{A}\mathbf{P}(x+S(1)\in dy,\tau_x>1)\frac{\mathbf{P}(\tau_y>n-1)}{\mathbf{P}(\tau_y>n)}\\
&=\int_{A}\mathbf{P}(x+S(1)\in dy,\tau_x>1)\frac{U(y)}{U(x)}\\
&=\frac{\mathbf{E}\left[U(x+S(1)),\tau_x>1,x+S(1)\in A\right]}{U(x)}.
\end{align*}
Then we can extend $\hat{P}(x,\cdot)$ to a finite measure on the Borel subsets of $W$. But this measure is not probabilistic, since
$$
\hat{P}(x,W)=\frac{\mathbf{E}\left[U(x+S(1)),\tau_x>1\right]}{U(x)}=\frac{U(x)-v(x)}{U(x)}<1.
$$

We loose the mass because of an ``infinite'' jump in the first step. Indeed, according to the optimal strategy in Theorem \ref{T1},
one of the random walks should have a jump of order $n^{1/2}$, and we let $n$ go to infinity. This infinite jump is the reason,
why a Markov chain, corresponding to the kernel $\hat{P}(x,A)$ has almost sure finite lifetime. Similar effects have been observed already
in other models. Bertoin and Doney \cite{BD94} have proven that a one-dimensional random walk with negative drift and regularly
varying tail conditioned to stay positive has finite lifetime. Jacka and Warren \cite{JW02} have shown that the same effect appears
in the Kolmogorov K2 chain. 

Having in mind this picture with ``infinite'' jumps, we can construct a conditioned random walk, which lives on the following set
$$
\hat{W}:=W\cup W_1\cup W_2,
$$
where
\begin{align*}
W_1&=\left\{(x_1,x_2,\ldots,x_{k-1},\infty),\,x_1<x_2<\ldots<x_{k-1}\right\}\\
W_2&=\left\{(-\infty,x_2,x_3,\ldots,x_k),\,x_2<x_3<\ldots<x_k\right\}.
\end{align*}
We define the transition probability by the following relations:
\begin{enumerate}
\item If $x\in W$ and $A\subset W$, then 
$$
\hat{P}(x,A)=\frac{\mathbf{E}\left[U(x+S(1)),\tau_x>1,x+S(1)\in A\right]}{U(x)}.
$$
\item If $x\in W$ and $A=A'\times\{\infty\}\subset W_1$, then
$$
\hat{P}(x,A)=\frac{p\mathbf{E}\left[v_1(x+S(1)),\tau_x^{(1)}>1,x+S(1)\in A'\right]}{U(x)}.
$$
\item If $x\in W$ and $A=\{-\infty\}\times A'\subset W_2$, then
$$
\hat{P}(x,A)=\frac{q\mathbf{E}\left[v_2(x+S(1)),\tau_x^{(2)}>1,x+S(1)\in A'\right]}{U(x)}.
$$
\item If $x\in W_1$ and $A=A'\times\{\infty\}\subset W_1$, then
$$
\hat{P}(x,A)=\frac{\mathbf{E}\left[v_1(x+S(1)),\tau_x^{(1)}>1,x+S(1)\in A'\right]}{v_1(x)}.
$$
\item If $x\in W_2$ and $A=\{-\infty\}\times A'\subset W_2$, then
$$
\hat{P}(x,A)=\frac{\mathbf{E}\left[v_2(x+S(1)),\tau_x^{(2)}>1,x+S(1)\in A'\right]}{v_2(x)}.
$$
\end{enumerate}
Here
$$
\tau^{(i)}_x:=\min\{n\geq1:\,x+S^{(i)}\notin W\},\ i=1,2
$$
and
$$
S^{(1)}:=(S_1,S_2,\ldots,S_{k-1})\text{ and }S^{(2)}:=(S_2,S_3,\ldots,S_k).
$$

The asymptotic behaviour of the corresponding Markov chain, say $\{\hat{S}(n),n\geq0\}$, 
can be described as follows. 
One of the random walks jumps away at time $m$ with probability 
$\mathbf{E}\left[v(x+S(m-1)),\tau_x>m-1\right]/U(x)$. Then we restart our process, which 
has from now on one frozen coordinate, either $-\infty$ or $\infty$, and $k-1$ ordered 
random walks. But for $k-1$ random walks we can apply Theorem 3 of \cite{DW10}. As a result 
we have that the limit of $\left\{\hat{S}([nt])/\sqrt{n},t\in[r_n/n,1]\right\}$ converges 
weakly to a process $\{X(t),t\in(0,1]\}$, where $r_n$ is such that $r_n\to\infty$. (We need 
this additional restriction because of jumps at bounded times.) The limit can be constructed 
as follows: Let $D(t)$ denote here the $(k-1)$-dimensional Dyson Brownian motion starting 
from zero. With some probability $p(x)$ we add to $D(t)$ one coordinate with constant value 
$\infty$, and with probability $q(x)=1-p(x)$ we add the coordinate with value $-\infty$. 

We have constructed a model of ordered random walks on an enlarged state space by formalising 
an intuitive picture of big jumps. But it remains unclear whether one can find a 
harmonic function for the substochastic kernel $\mathbf{P}(x+S(1)\in dy,\tau_x>1)$. If such 
a function exists, then one can construct a model of ordered random walks on the original
Weyl chamber. We conjecture that there are no harmonic functions for ordered random walks with
heavy tails. The examples from \cite{BD94,JW02}, which we have mentioned above, support this conjecture.

The most standard way to describe the set of harmonic functions consists in the study
of the corresponding Martin boundary. We found only a few results on Martin boundary for killed 
random walks. Doney \cite{Don98} found sufficient and necessary conditions 
for existence of harmonic functions in one-dimensional case. 
The proof relies on the Wiener-Hopf factorisation, which seems to work in the 
one-dimensional case only.
In a series of papers \cite{IR08, IRL10,IR10} by Ignatiouk-Robert, and by Ignatiouk-Robert and Loree
Martin boundaries for killed random walks with non-zero drift in a half-space and in a quadrant have 
been studied. In all these papers the Cramer condition has been imposed.
Next-neighbour random walks with zero mean in the Weyl chamber have been studied by Raschel \cite{Rasch09a,Rasch09b}.
In our situation all the increments are heavy-tailed. This means that one needs another method for finding the 
Martin boundary.
\subsection{Conditional limit theorem for $S$}
In this paragraph we turn our attention to the behaviour of $\{S([nt])/\sqrt{n},\,t\leq1\}$ conditioned on $\{\tau_x>n\}$.
Since one of the random walks should have a jump of order $\sqrt{n}$ on the event $\{\tau_x>n\}$, 
this conditioning will not lead to an infinite jump, as it happens in the case of conditioning on $\{\tau_x=\infty\}$.

We define
$$
X^{(n)}(t):=\frac{x+S([nt]\wedge r_n)}{\sqrt{n}},\quad t\in[0,1].
$$
Here $r_n\to\infty$ and $r_n=o(n)$. (Again, we need to go away from zero, because of a big jump occurring at the
very beginning.) In order to state our limit theorem we have to introduce a limiting process, say $X$.
The distribution of the starting point, $X(0)$, is given by
$$
\mu_x(dy)=q(x)f(-y_1)dy_1\prod_{i=2}^k\delta_0(dy_i)+p(x)f(y_k)dy_k\prod_{i=1}^{k-1}\delta_0(dy_i),
$$
where $f(x)=\theta^{-1}\psi(x)x^{-\alpha-1}{\rm 1}_{\mathbb{R}_+}(x)$ with $\psi$ defined in (\ref{L4.5}), and
$$
p(x):=\frac{p\sum_{l=0}^{\infty}\mathbf{E}\left[v_1(x+S(l)),\tau_x>l\right]}{U(x)},\ 
q(x):=\frac{q\sum_{l=0}^{\infty}\mathbf{E}\left[v_2(x+S(l)),\tau_x>l\right]}{U(x)}.
$$
Further, given $X(0)=y$, we define
$$
\mathcal{L}\left(X\right)=\lim_{a\to0}
\mathcal{L}\left(y(a)+B(t),t\in[0,1]\Big|\tau_{y(a)}^{bm}>1\right), 
$$
where $y(a)=y+a(0,1,2,\ldots,k-1)$.
\begin{theorem}
\label{T2}
Under the conditions of Theorem \ref{T1},
$$
\{X^{(n)}|\tau_x>n\}\Rightarrow X
$$
in the Skorohod topology on $C[0,1]$.
\end{theorem}
It is worth mentioning that the limiting process is not invariant with respect
to the starting position of the random walk. More precisely, the distribution 
of $X(0)$ depends on $x$ through $p(x)$ and $q(x)$.  Clearly this happens beacuse of one large jump 
i the beginning. An analogous result can be 
proven also for random walks with $\mathbf{E}|X|^{k-1}<\infty$, but the limiting 
process will start always at zero.
\subsection{Some remarks on the general case.} 
Although the informal picture behind Theorems \ref{T1} and \ref{T2} is quite simple,
the proofs are very technical. In the case of smaller values of $\alpha$, i.e.
$\alpha<k-2$, one has to overcome even more technical difficulties, which are of 
the combinatorial nature. However, is it clear that our approach works in the case 
$\alpha<k-2$ as well. In this paragraph we describe the behaviour
of ordered random walks for such values of $\alpha$. 

First, in order to stay in $W$ at least up to time $n$,
the random walk $S$ should have $k_\alpha:=k-[\alpha+1]$ big jumps. Then it may happen that
at least two jumps go in the same direction (upwards or downwards). The values of all
these jumps should be ordered. As a result one gets the following relation:
$$
\mathbf{P}(\tau_x>n)\sim\ U(x)n^{-\alpha k_\alpha/2}n^{-(k-k_\alpha)(k-k_\alpha-1)/4}
$$
with some superharmonic function $U$. Second, to construct ordered random walks
we need to add all vectors with $k_\alpha$ infinite coordinates. Finally,
in Theorem~\ref{T2} one has to change the distribution of $X_0$ only: The 
limiting process will start from a random point with $k_\alpha$ non-zero coordinates.

Unfortunately, the case of integer values of $\alpha$ remains unsolved. If,
for example, $\alpha=k-1$, then, the jumps of order $\sqrt{n}$ do not contribute
to $\mathbf{P}(\tau_x>n)$. Therefore, we can not use the method proposed in the
present work.

If $\alpha<2$, then one can describe the asymptotic behaviour of $\tau_x$. If
$k=2$, then $\mathbf{P}(\tau_x>n)$ is of order $n^{-1/2}$. And if $k\geq 3$, then
we expect $k-2$ big jumps. But in this situation all big jumps are of order
$n^{1/\alpha}$. As a result one will obtain
$$
\mathbf{P}(\tau_x>n)\sim\ U(x)n^{-k+3/2}.
$$

Our last remark concerns other Weyl chambers. K\"onig and Schmid \cite{KS10} have shown that the
approach proposed in \cite{DW10} works also in Weyl chambers of types $C$ and $D$. It is easy to see
that, using the method from the present paper, one can prove analogons of Theorems \ref{T1} and \ref{T2} 
for chambers of types $C$ and $D$. Moreover, since big negative jumps lead to exit from these two regions,
the corresponding optimal strategies are even simpler then in the chamber of type $A$.


\section{Finiteness of the superharmonic function}
\begin{proposition}\label{P1}
Under the assumptions of Theorem~\ref{T1},
$$
U(x)=\sum_{l=0}^\infty\mathbf{E}\left[v(x+S(l)),\tau_x>l\right]<\infty.
$$
\end{proposition}
We first introduce some notation. For every $\varepsilon>0$ denote
$$
W_{n,\varepsilon}:=\left\{x\in\mathbb{R}^k:\,|x_j-x_i|>n^{1/2-\varepsilon},1\leq i<j\leq k\right\}
$$
and let 
$$
\nu_n:=\min\{j\geq1:x+S(j)\in W_{n,\varepsilon}\}
$$
be the first time the random walk enters this region. 
\begin  {proof}
Fix $\delta>0$. Let $\eta^{\pm}$ be the moments of 'big' jumps upwards and downwards, 
i.e.,
$$
\eta^+=\min\left\{l\geq1: X_k(l)>n^{(1-\delta)/2}\right\}\text{ and }
\eta^-=\min\left\{l\geq1: X_1(l)<-n^{(1-\delta)/2}\right\}.
$$
Let $\eta=\min\{\eta^+,\eta^-\}$ be the first big jump.  

First we note that 
\begin{align} 
\mathbf{E}\left[v(x+S(n)),\tau_x>n\right]&=
\mathbf{E}\left[v(x+S(n)),\tau_x>n,\nu_n\le n^{1-\varepsilon}\right]\nonumber\\
&+\mathbf{E}\left[v(x+S(n)),\tau_x>n,\nu_n>n^{1-\varepsilon}\right]\label{eq:1}. 
\end{align}
To estimate the second term we apply Proposition 4 of \cite{DW10} 
to obtain 
$$
c_*\Delta_1^{(i)}(x)\leq v_i(x)\leq c^*\Delta_1^{(i)}(x),\quad i=1,2,
$$
where
$$
\Delta_1^{(1)}(x):=\prod_{1\leq i<j\leq k-1}(1+|x_j-x_i|)\quad\text{and}\quad
\Delta_1^{(2)}(x):=\prod_{2\leq i<j\leq k}(1+|x_j-x_i|).
$$
Then, according to Lemma 8 in \cite{DW10},
\begin{align}
\label{L1.1}
\nonumber
&\mathbf{E}\left[v(x+S(n)),\tau_x>n,\nu_n> n^{1-\varepsilon}\right]\\
\nonumber
&\leq p\mathbf{E}\left[v_1(x+S(n)),\tau_x^{(1)}>n,\nu_n> n^{1-\varepsilon}\right]+
q\mathbf{E}\left[v_2(x+S(n)),\tau_x^{(2)}>n,\nu_n> n^{1-\varepsilon}\right]\\
\nonumber
&\leq C\left(\Delta_1^{(1)}(x)+\Delta_1^{(2)}(x)\right)\exp\{-C n^\varepsilon\}\\
&\leq Cv(x)\exp\{-C n^\varepsilon\}.
\end{align}
This gives us an estimate for the second term of (\ref{eq:1}). 

The rest of the proof is devoted to estimation of the first summand in 
(\ref{eq:1}). We split this term in three parts: with big jump upwards, big jump downwards and 
no big jumps,  
\begin{align*}
&\hspace{-1cm}\mathbf{E}\left[v(x+S(n)),\tau_x>n,\nu_n\le n^{1-\varepsilon}\right]\\
&\leq\mathbf{E}\left[v(x+S(n)),\tau_x>n,\eta^+\leq\nu_n\leq n^{1-\varepsilon}\right]\\
&\hspace{0.5cm}+\mathbf{E}\left[v(x+S(n)),\tau_x>n,\eta^-\le \nu_n\leq n^{1-\varepsilon},\eta^+>\nu_n\right]\\
&\hspace{0.5cm}+\mathbf{E}\left[v(x+S(n)),\tau_x>n,\nu_n\le n^{1-\varepsilon},\eta>\nu_n\right]\\
&=:E_{up}+E_{down}+E_{no}.
\end{align*}
We construct estimates for each of terms separately and then combine them. 
We apply the resulting estimate recursively several times and prove the claim.  

\vspace{12pt}

{\bf Big jump upwards:}
Using the Markov property, we get
\begin{align*}
E_{up}&=\sum_{l=1}^{n^{1-\varepsilon}}\int_W\mathbf{P}(x+S(l)\in dy,\tau_x>l,\eta^+=l,\nu_n\ge l)\\
&\hspace{2cm}\times\mathbf{E}[v_1(y+S(n-l)),\tau_y>n-l,\nu_n\le n^{1-\varepsilon}-l].
\end{align*}
We apply Proposition 4 of \cite{DW10} to the system of $k-1$ random walks,  
$$
\sup_{n\geq0}\mathbf{E}[v_1(y+S(n)),\tau_y>n]\leq \sup_{n\geq0}\mathbf{E}[v_1(y+S(n)),\tau_y^{(1)}>n]\leq Cv_1(y).
$$
Therefore,
$$
\mathbf{E}[v_1(y+S(n-l)),\tau_y>n-l]\leq Cv_1(y)
$$
and, consequently,
\begin{align*}
E_{up}
\leq C\sum_{l=1}^{n^{1-\varepsilon}}\mathbf{E}\left[v_1(x+S(l)),\tau_x>l,\eta^+=l\right].
\end{align*}
Using the Markov property once again, we have
\begin{align*}
&\mathbf{E}\left[v_1(x+S(l)),\tau_x>l,\eta^+=l\right]\\
&=\int_W \mathbf{P}(x+S(l-1)\in dy,\tau_x>l-1,\eta^+>l-1)\\
&\hspace{2cm}\times\mathbf{E}\left[v_1(y+X(1)),\tau_y>1,X_k(1)>n^{(1-\delta)/2}\right].
\end{align*}
The random variable $X_k$ is independent of 
$X_1,\ldots,X_{k-1}$, 
\begin{align*}
\mathbf{E}&\left[v_1(y+X(1)),\tau_y>1,X_k(1)>n^{(1-\delta)/2}\right]\\
&\le 
\mathbf{E}\left[v_1(y+X(1)),\tau^{(1)}_y>1, X_k(1)>n^{(1-\delta)/2}\right]\\
&=\mathbf{E}\left[v_1(y+X(1)),\tau_y^{(1)}>1\right]\mathbf P\left(X_k(1)>n^{(1-\delta)/2}\right).
\end{align*}
Hence,
\begin{equation*}
\mathbf{E}\left[v_1(x+S(l)),\tau_x>l,\eta^+=l\right]
\le 
pn^{-\alpha(1-\delta)/2} \mathbf{E}\left[v_1(x+S(l-1)),\tau_x>l-1\right].
\end{equation*}
Summing up over $l\leq n^{1-\varepsilon}$, we obtain
\begin{align}
\label{L1.2}
E_{up}\leq Cn^{-\alpha(1-\delta)/2}\sum_{l=1}^{n^{1-\varepsilon}}\mathbf{E}\left[v_1(x+S(l-1)),\tau_x>l-1\right].
\end{align}

\vspace{12pt}

{\bf Big jump downwards:}
We now turn our attention to the case when all jumps of the random walk on the top are bounded by $n^{(1-\delta)/2}$. 
First of all we note that according to one of Fuk-Nagaev inequalities, see Corollary 1.11 in \cite{Nag79},
\begin{equation}\label{L1.2a}
\mathbf{P}\left(\max_{j\leq n^{1-\varepsilon}}\left[S_k(j){\rm 1}\{\eta^+>j\}\right]>n^{1/2-r(\delta)}\right)
\leq \exp\{-C n^{\delta^2/2}\},
\end{equation}
where $r(\delta)=\delta/2-\delta^2/2$.
This yields
\begin{align}
\label{L1.3}
\nonumber
&\mathbf{E}\left[v_1(x+S(n)),\tau_x>n,\max_{j\leq\nu_n}S_k(j)>n^{1/2-r(\delta)},\nu_n\leq n^{1-\varepsilon}, \eta^+>\nu_n\right]\\
\nonumber
&\hspace{1cm} \leq\mathbf{E}\left[v_1(x+S(n)),\tau_x^{(1)}>n,\max_{j\leq n^{1-\varepsilon}}\left[S_k(j){\rm 1}\{\eta^+>j\}\right]>n^{1/2-r(\delta)}\right]\\
\nonumber
&\hspace{1cm} =\mathbf{E}\left[v_1(x+S(n)),\tau_x^{(1)}>n\right]\mathbf{P}\left(\max_{j\leq n^{1-\varepsilon}}\left[S_k(j){\rm 1}\{\eta^+>j\}\right]>n^{1/2-r(\delta)}\right)\\
&\hspace{1cm}\leq v_1(x)\exp\{-C n^{\delta^2/2}\}.
\end{align}
Next we need to analyse the case when the top random walk is 
always less than $n^{1/2}$. Hence,
\begin{align*}
E_{down}&\le v_1(x)\exp\{-C n^{\delta^2/2}\}\\
&+\sum_{l=1}^{n^{1-\varepsilon}}
\mathbf{E}\left[v_1(x+S(n)),\tau_x>n,\eta^-=l\leq\nu_n,\nu_n\leq n^{1-\varepsilon},\max_{j\leq\nu_n}S_k(j)<n^{1/2},\eta^+>\nu_n\right]\\
&=v_1(x)\exp\{-C n^{\delta^2/2}\}+\sum_lE_{down,l}. 
\end{align*}
Clearly in the definition of $E_{down,l}$ the big jum occurs at time  $l$. Also note that 
we have excluded the possibility that the top random walk  
goes up without a big jump. 
Applying the Markov property again, 
\begin{align*}
E_{down,l}&=
\int_W\mathbf{P}\left(x+S(l)\in dy,\tau_x>l,\eta^-=l\right)\nonumber\\
&\times
\mathbf{E}\left[v_1(y+S(n-l)),\tau_y>n-l,\nu_n<n^{1-\varepsilon}-l,\max_{j\leq\nu_n}S_k(j)<n^{1/2}\right]
\nonumber
\\
&=:
\int_W\mathbf{P}\left(x+S(l)\in dy,\tau_x>l,\eta^-=l\right) 
E_{after,l}(y).
\end{align*}
Using the Markov property for the multiplier, 
\begin{align*}
E_{after,l}(y)
=\sum_{r=1}^{n^{1-\varepsilon}-l}\int_{W_{n,\varepsilon}}
&\mathbf{P}\left(y+S(r)\in dz,\tau_y>r=\nu_n, \max_{j\leq r}S_k(j)<n^{1/2}\right)\\
\times &\mathbf{E}\left[v_1(z+S(n-l-r)),\tau_z>n-l-r\right].
\end{align*}
It follows from the martingale property of $v_1$ that
\begin{align*}
&\mathbf{E}\left[v_1(z+S(n-l-r)),\tau_z>n-l-r\right]\\
&\hspace{1cm}\leq \mathbf{E}\left[v_1(z+S(n-l-r)),\tau_z^{(1)}>n-l-r\right]=v_1(z).
\end{align*}
Consequently,
\begin{align}
  \label{eq:2}
E_{after,l}(y)\leq
\mathbf{E}\left[v_1(y+S_{\nu_n}),\tau_y>\nu_n,\max_{j\le \nu_n}S_{k}(j)<n^{1/2}\right].
\end{align}
It follows from Proposition 4 of \cite{DW10} that, uniformly in $z\in W_{n,\varepsilon}$,
\begin{align*}
v_1(z)&\leq C\prod_{1\leq i<j\leq k-1}(z_j-z_i)
\\
&\leq C\frac{(z_k-z_1)^{k-2}}{\prod_{2\leq l\leq k-1}(z_k-z_l)}\prod_{2\leq i<j\leq k}(z_j-z_i)\\
&\leq C(z_k-z_1)^{k-2} n^{-(1/2-\varepsilon)(k-2)}v_2(z).
\end{align*}
Therefore, since  $S_{\nu_n}\in W_{n,\varepsilon}$ and $S_k(\nu_n)<n^{1/2}$ 
it follows from the latter inequality and (\ref{eq:2}) that 
 \begin{align*}
&E_{after,l}(y)\\
&\le Cn^{-(1/2-\varepsilon)(k-2)}\mathbf{E}\left[(n^{1/2}-S_1(\nu_n))^{k-2}v_2(y+S_{\nu_n}),\tau_y>\nu_n,\nu_n\leq n^{1-\varepsilon}-l\right]\\
&\leq Cn^{-(1/2-\varepsilon)(k-2)}\\
&\hspace{0.5cm}\times\mathbf{E}\left[\left(n^{1/2}-y_1-M_1(n^{1-\varepsilon})\right)^{k-2}v_2(y+S(\nu_n)),
\tau_y^{(2)}>\nu_n,\nu_n\leq n^{1-\varepsilon}-l\right],
\end{align*}
where $M_1(n):=\min_{k\leq n}S_1(k)$.
Using now the fact that the sequence  $v_2(y+S_n){\rm 1}\{\tau^{(2)}_y>n\}$ is a martingale, 
we get
\begin{align*}
&E_{after,l}(y)
\\
&\quad\leq Cn^{-(1/2-\varepsilon)(k-2)}
\mathbf{E}\left[\left(n^{1/2}-y_1-M_1(n^{1-\varepsilon})\right)^{k-2}v_2(y+S(n)),\tau_y^{(2)}>n\right]\\
&\quad =Cn^{-(1/2-\varepsilon)(k-2)}
\mathbf{E}\left[\left(n^{1/2}-y_1-M_1(n^{1-\varepsilon})\right)^{k-2}\right]\mathbf{E}\left[v_2(y+S(n)),\tau_y^{(2)}>n\right]\\
&\quad =Cn^{-(1/2-\varepsilon)(k-2)}
\mathbf{E}\left[\left(n^{1/2}-y_1-M_1(n^{1-\varepsilon})\right)^{k-2}\right]v_2(y).
\end{align*}
Applying the Rosenthal inequality, we get finally
\begin{align*}
&
E_{after,l}(y)
\leq Cn^{-(1/2-\varepsilon)(k-2)} v_2(y)\left(|y_1|^{k-2}+n^{(k-2)/2}\right).
\end{align*}
Using this bound, we get
\begin{align*}
&E_{down,l}\\
&\leq C n^{-(1/2-\varepsilon)(k-2)}
\int_W\mathbf{P}\left(x+S(l)\in dy,\tau_x>l,\eta^-=l\right)v_2(y)\left(|y_1|^{k-2}+n^{(k-2)/2}\right)\\
&=C n^{-(1/2-\varepsilon)(k-2)}\mathbf{E}\left[\left(|x+S_1(l)|^{k-2}+n^{(k-2)/2}\right)v_2(x+S(l)),\tau_x>l,\eta^-=l\right]
\end{align*}

We split the latter expectation in two parts. 
First on the event $\{S_1(l-1)\geq-n^{1/2}\}$ we have 
\begin{align}\label{L1.4}
\nonumber
&\mathbf{E}\left[|x+S_1(l)|^{k-2}v_2(x+S(l)),\eta^-=l,\tau_x>l,S_1(l-1)\geq-n^{1/2}\right]\\
\nonumber
&\quad\leq C\mathbf{E}\left[\left((-X_1(l))^{k-2}+n^{(k-2)/2}\right)v_2(x+S(l)),\tau_x>l,\eta^-=l\right]\\
\nonumber
&\quad\leq C\mathbf{E}\left[v_2(x+S(l-1)),\tau_x>l-1\right]\\
\nonumber
&\hspace{2cm}\times
\mathbf{E}\left[\left((-X_1(l))^{k-2}+n^{(k-2)/2}\right),X_1(l)<-n^{(1-\delta)/2}\right]\\
&\quad\leq C n^{(k-2)/2-\alpha(1-\delta)/2}\mathbf{E}\left[v_2(x+S(l-1)),\tau_x>l-1\right].
\end{align}

Second the probability of event $\{S_1(l-1)<-n^{1/2}\}$ 
is negligible due to the Fuk-Nagaev inequality, 
$$
\mathbf{P}\left(S_1(l-1)<-z,\eta^->l-1\right)\leq\exp\left\{-C z/n^{(1-\delta)/2}\right\},\quad z>n^{1/2}.
$$
Therefore, in view of the martingale property of $v_2(y+S_n)1\{\tau^{(2)}_y>n\}$,
\begin{align*}
&\mathbf{E}\left[|x+S_1(l-1)|^{k-2}v_2(x+S(l)),\eta^-=l,\tau_x>l,S_1(l-1)<-n^{1/2}\right]\\
&\quad\leq v_2(x)
\mathbf{E}\left[|x+S_1(l-1)|^{k-2}\eta^->l-1,S_1(l-1)<-n^{1/2}\right]\\
&\quad\leq v_2(x)\exp\left\{-C n^{-\delta/2}\right\}.
\end{align*}
Combining the latter estimate with (\ref{L1.4}) and 
using a bound 
$$
|x+S_1(l)|^{k-2}\leq 2^{k-3}\left(|x+S_1(l-1)|^{k-2}+(-X_1(l))^{k-2}\right),
$$
we get
\begin{align}\label{L1.5}
\nonumber
&\mathbf{E}\left[|x+S_1(l-1)|^{k-2}v_2(x+S(l)),\eta^-=l,\tau_x>l\right]\\
&\leq v_2(x)\exp\left\{-C n^{-\delta/2}\right\}+
C n^{(k-2)/2-\alpha(1-\delta)/2}\mathbf{E}\left[v_2(x+S(l-1)),\tau_x>l-1\right].
\end{align}
>From (\ref{L1.4}) and (\ref{L1.5}) we conclude
\begin{equation*}
E_{down, l}
\leq v_2(x)\exp\left\{-C n^{-\delta/2}\right\}+
C n^{-\alpha/2+\delta_1}\mathbf{E}\left[v_2(x+S(l-1)),\tau_x>l-1\right],
\end{equation*}
where $\delta_1=\varepsilon(k-2)+\alpha\delta/2$.
Summing up over $l$ and taking into account (\ref{L1.3}), we obtain
\begin{align}
\label{L1.6}
\nonumber
{E_{down}}
&\leq v_2(x)\exp\left\{-C n^{-\delta/2}\right\}\\
&+Cn^{-\alpha/2+\delta_1}\sum_{l=1}^{n^{1-\varepsilon}}\mathbf{E}\left[v_2(x+S(l-1)),\tau_x>l-1\right].
\end{align}

\vspace{12pt}

{\bf No big jumps:}
It remains to consider the case with no big jumps before the stopping time $\nu_n$. 
If all the jumps are bounded by $n^{(1-\delta)/2}$, then, as it was shown 
in the proof of Lemma 16 of \cite{DW10},
\begin{align}
\label{L1.7}
\nonumber
&\mathbf{E}\left[v_1(x+S(n)),\tau_x>n,|S(\nu_n)|>n^{1/2-\delta/4},\eta\geq\nu_n, \nu_n\leq n^{1-\varepsilon}\right]\\
&\hspace{2cm}\leq
C\exp\left\{-C n^{-\delta/4}\right\}.
\end{align}

If the random walk starts from $y\in W_{n,\varepsilon}$ with $|y|\leq n^{1/2-\delta/4}$, 
then one can use the standard KMT-coupling to show that
\begin{align}
\label{L1.7a}
\nonumber
\mathbf{E}\left[v_1(y+S(n)),\tau_y>n\right]&\sim \mathbf{E}\left[\Delta^{(1)}(y+S(n)),\tau_y>n\right]\\
\nonumber
&\sim\mathbf{E}\left[\Delta^{(1)}(y+B(n)),\tau^{bm}_y>n\right]\\
&\sim \frac{\Delta^{(1)}(y)} {n^{(k-1)/2}}\mathbf{E}\left[\Delta^{(1)}(B(1))|\tau^{bm}_{y/\sqrt{n}}>1\right].
\end{align}
Moreover, if $\gamma$ is sufficiently small, $\gamma<\delta/8$, then, using the same arguments,
\begin{align}\label{L1.7b}
\nonumber
\mathbf{E}\left[v_1(y+S(n^{1-\gamma})),\tau_y>n^{1-\gamma}\right]&\sim
\mathbf{E}\left[\Delta^{(1)}(y+S(n^{1-\gamma})),\tau_y>n^{1-\gamma}\right]\\
\nonumber
&\sim\mathbf{E}\left[\Delta^{(1)}(y+B(n^{1-\gamma})),\tau^{bm}_y>n^{1-\gamma}\right]\\
&\sim \frac{\Delta^{(1)}(y)} {n^{(1-\gamma)(k-1)/2}}\mathbf{E}\left[\Delta^{(1)}(B(1))|\tau^{bm}_{y/n^{(1-\gamma)/2}}>1\right].
\end{align}
Since $y/n^{(1-\gamma)/2}\to 0$, we have
$$
\mathbf{E}\left[\Delta^{(1)}(B(1))|\tau^{bm}_{y/\sqrt{n}}>1\right]\sim
\mathbf{E}\left[\Delta^{(1)}(B(1))|\tau^{bm}_{y/n^{(1-\gamma)/2}}>1\right].
$$

>From this relation, estimates (\ref{L1.7a}) and (\ref{L1.7b}), and the strong Markov property we infer that
\begin{align*}
&\mathbf{E}\left[v_1(x+S(n)),\tau_x>n,|S(\nu_n)|\leq n^{1/2-\delta/4},\eta\geq\nu_n, \nu_n\leq n^{1-\varepsilon}\right]\\
&\sim n^{-\gamma(k-1)/2}\mathbf{E}\left[v_1(x+S(n^{1-\gamma})),\tau_x>n^{1-\gamma},|S(\nu_n)|\leq n^{1/2-\delta/4},\eta\geq\nu_n, \nu_n\leq n^{1-\varepsilon}\right].
\end{align*}
Hence
\begin{align}\label{L1.8}
\nonumber
\mathbf{E}\left[v_1(x+S(n)),\tau_x>n,|S(\nu_n)|\leq n^{1/2-\delta/4},\eta\geq\nu_n, \nu_n\leq n^{1-\varepsilon}\right]\\
\leq C n^{-\gamma(k-1)/2}\mathbf{E}\left[v_1(x+S(n^{1-\gamma})),\tau_x>n^{1-\gamma} \right].
\end{align}

\vspace{12pt}

{\bf Final recursion:}
Putting (\ref{L1.1}), (\ref{L1.2}), (\ref{L1.6})--(\ref{L1.8}) together, we obtain
\begin{align*}
\mathbf{E}\left[v_1(x+S(n)),\tau_x>n\right]\leq
Cn^{-\alpha/2+\delta_1}\sum_{l=1}^{n^{1-\varepsilon}}\mathbf{E}\left[v(x+S(l-1)),\tau_x>l-1\right]\\
+C n^{-\gamma(k-1)/2}\mathbf{E}\left[v_1(x+S(n^{1-\gamma})),\tau_x>n^{1-\gamma} \right]+C\exp\{-Cn^{\delta/4}\}.
\end{align*}
Because of the symmetry, an analogous bound holds for $\mathbf{E}\left[v_2(x+S(n)),\tau_x>n\right]$.
Consequently,
\begin{align*}
\mathbf{E}\left[v(x+S(n)),\tau_x>n\right]\leq Cn^{-\alpha/2+\delta_1}
\sum_{l=1}^{n^{1-\varepsilon}}\mathbf{E}\left[v(x+S(l-1)),\tau_x>l-1\right]\\
+C n^{-\gamma(k-1)/2}\mathbf{E}\left[v(x+S(n^{1-\gamma})),\tau_x>n^{1-\gamma} \right]+C\exp\{-Cn^{\delta/4}\}.
\end{align*}
Iterating this bound $N$ times and recalling that $\alpha<k-1$, we obtain
\begin{align*}
\mathbf{E}\left[v(x+S(n)),\tau_x>n\right]\leq Cn^{-\alpha/2+\delta_1}\sum_{l=1}^{n}\mathbf{E}\left[v(x+S(l-1)),\tau_x>l-1\right]\\
+C n^{-(1-(1-\gamma)^N)(k-1)/2}v(x)+C\exp\{-Cn^{(1-\gamma)^{N-1}\delta/4}\}.
\end{align*}
If $N$ is such that $(1-(1-\gamma)^N)(k-1)/>\alpha$, then
\begin{align}\label{L1.9}
\nonumber
&\mathbf{E}\left[v(x+S(n)),\tau_x>n\right]\\
&\hspace{1cm}\leq Cn^{-\alpha/2+\delta_1}\sum_{l=1}^{n}\mathbf{E}\left[v(x+S(l-1)),\tau_x>l-1\right]
+C(x) n^{-\alpha/2}.
\end{align}
We know that $\mathbf{E}\left[v(x+S(l-1)),\tau_x>l-1\right]\leq v(x)$. Entering with this into (\ref{L1.9}), we get
\begin{align}\label{L1.10}
\mathbf{E}\left[v(x+S(n)),\tau_x>n\right]\leq C(x)n^{1-\alpha/2+\delta_1}.
\end{align}
If $\alpha>4$, then making $\delta_1$ sufficiently small, we see that $\mathbf{E}\left[v(x+S(n)),\tau_x>n\right]$ is summable.
If $\alpha\leq4$, then applying (\ref{L1.10}) to every expectation on the right hand side of (\ref{L1.9}), we get
$$
\mathbf{E}\left[v(x+S(n)),\tau_x>n\right]\leq C(x)n^{2-\alpha+2\delta_1}.
$$
We are done if $\alpha>3$. If it is not the case, then we enter with the new bound into (\ref{L1.9}), and so on.
The $N$-th iteration will give the bound of order $n^{N(1-\alpha/2+\delta_1)}$. If $N(1-\alpha/2+\delta_1)<-1$, then
we have the desired summability.
\end{proof}

\section{Proof of Theorem \ref{T1}}
We start by estimating the tail of $\tau_x$ for paths without big jumps.
\begin{lemma}\label{L2}
Let $A_n(a)$ denote the event $\{X_1(l)\geq-an^{1/2},X_k(l)\leq an^{1/2}\text{ for all }l\leq n\}$.
Then
$$
\mathbf{P}\left(\tau_x>n,A_n(a)\right)\leq Ca^{k-1-\alpha} n^{-\alpha/2-(k-1)(k-2)/4}.
$$
\end{lemma}
\begin{proof}
First, if $\eta>\nu_n$, then, repeating coupling arguments from the proof of Proposition \ref{P1}, we obtain
\begin{align}
\label{L2.1}
\nonumber
&\mathbf{P}\left(\tau_x>n,A_n(a),\eta>\nu_n,\nu_n\leq n^{1-\varepsilon}\right)\\
\nonumber
&\quad\sim n^{-\gamma k(k-1)/4}\mathbf{P}\left(\tau_x>n^{1-\gamma},A_{n^{1-\gamma}}(a),\eta>\nu_n,\nu_n\leq n^{1-\varepsilon}\right)\\
&\quad\leq n^{-\gamma k(k-1)/4}\mathbf{P}\left(\tau_x>n^{1-\gamma},A_{n^{1-\gamma}}(a)\right).
\end{align}

We next assume that the random walk on the bottom jumps before $\nu_n$ but the random walk on the top do not jump, 
i.e. $\eta^-\leq \nu_n$ and $\eta^+> \nu_n$. It follows from (\ref{L1.2a}) that
\begin{align*}
&\mathbf{P}\left(\tau_x>n, A_n(a),\eta^-\leq\nu_n<\eta^+,\nu_n\leq n^{1-\varepsilon}\right)\\
&\quad\leq \mathbf{P}\left(\tau_x>n, A_n(a),\eta^-\leq\nu_n\leq n^{1-\varepsilon}, S_k(\nu_n)\leq n^{1/2-r(\delta)}\right)
+\exp\{-Cn^{\delta^2/2}\}.
\end{align*}
Applying now estimate (33) from \cite{DW10}, we get
\begin{align*}
&\mathbf{P}\left(\tau_x>n, A_n(a),\eta^-\leq\nu_n\leq n^{1-\varepsilon}, S_k(\nu_n)\leq n^{1/2-r(\delta)}\right)\\
&\leq \frac{C}{n^{k(k-1)/4}}\mathbf{E}
\left[\Delta(x+S(\nu_n)),\tau_x>\nu_n, A_{\nu_n}(a),\eta^-\leq\nu_n\leq n^{1-\varepsilon}, S_k(\nu_n)\leq n^{1/2-r(\delta)}\right].
\end{align*}
On $W_{n,\varepsilon}$ holds
$$
\Delta(x+S(\nu_n))\leq \left(x_k-x_1+S_k(\nu_n)-S_1(\nu_n)\right)^{k-1}v_1(x+S(\nu_n)).
$$
Consequently,
\begin{align*}
&\mathbf{E}
\left[\Delta(x+S(\nu_n)),\tau_x>\nu_n, A_{\nu_n}(a),\eta^-\leq\nu_n\leq n^{1-\varepsilon}, S_k(\nu_n)\leq n^{1/2-r(\delta)}\right]\\
&\leq \mathbf{E}\left[\left(n^{1/2-r(\delta)}-S_1(\nu_n)\right)^{k-1}v_2(x+S(\nu_n)),
\tau_x>\nu_n, A_{\nu_n}(a),\eta^-\leq\nu_n\leq n^{1-\varepsilon}\right].
\end{align*}
Repeating arguments from the proof of Proposition \ref{P1}, we get
\begin{align*}
&\mathbf{E}
\left[\Delta(x+S(\nu_n)),\tau_x>\nu_n, A_{\nu_n}(a),\eta^-\leq\nu_n\leq n^{1-\varepsilon}, S_k(\nu_n)\leq n^{1/2-r(\delta)}\right]\\
&\leq \sum_{l=1}^{n^{1-\varepsilon}}\mathbf{E}\left[v_2(x+S(l-1)),\tau_x>l-1\right]
\mathbf{E}\left[ X^{k-1}, n^{(1-\delta)/2}\leq X\leq an^{1/2}\right]\\
&\leq a^{k-1-\alpha}n^{(k-1-\alpha)/2}\sum_{l=1}^{\infty}\mathbf{E}\left[v_2(x+S(l-1)),\tau_x>l-1\right]
\end{align*}
As a result we have
\begin{align}
\label{L2.2}
\nonumber
&\mathbf{P}\left(\tau_x>n, A_n(a),\eta^-\leq\nu_n<\eta^+,\nu_n\leq n^{1-\varepsilon}\right)\\
&\quad\quad\leq \frac{Ca^{k-1-\alpha}}{n^{\alpha/2+(k-1)(k-2)/4}}\sum_{l=1}^{\infty}\mathbf{E}\left[v_2(x+S(l-1)),\tau_x>l-1\right].
\end{align}
Analogously,
\begin{align}
\label{L2.3}
\nonumber
&\mathbf{P}\left(\tau_x>n, A_n(a),\eta^+\leq\nu_n<\eta^-,\nu_n\leq n^{1-\varepsilon}\right)\\
&\quad\quad\leq \frac{Ca^{k-1-\alpha}}{n^{\alpha/2+(k-1)(k-2)/4}}\sum_{l=1}^{\infty}\mathbf{E}\left[v_1(x+S(l-1)),\tau_x>l-1\right].
\end{align}

Therefore, it remains to consider the case when $\eta^+\leq\nu_n$ and $\eta^-\leq\nu_n$. Because of the symmetry we may assume that
$\eta^+\leq\eta^-$. Then
\begin{align*}
&\mathbf{P}\left(\tau_x>n,\eta^+\leq\eta^-\leq\nu_n\leq n^{1-\varepsilon}\right)\\
&\quad\leq\sum_{l=1}^{n^{1-\varepsilon}}\mathbf{P}\left(\tau_x>n,\eta^+=\eta^-=l\right)+
\sum_{l=1}^{n^{1-\varepsilon}}\sum_{j=l+1}^{n^{1-\varepsilon}}\mathbf{P}\left(\tau_x>n,\eta^+=l,\eta^-=j\right).
\end{align*}
First we note
$$
\mathbf{P}\left(\tau_x>n,\eta^+=\eta^-=l\right)\leq C\int_W\mathbf{P}(x+S(l-1)\in dy,\tau_x>l-1)n^{-\alpha+\delta}
\frac{\tilde{v}(y)}{n^{(k-2)(k-3)/4}},
$$
where $\tilde{v}$ is the invariant function for random walks $(S_2,\ldots,S_{k-1})$. Using now the bound
\begin{align}\label{L2.4}
\mathbf{E}\left[\tilde{v}(x+S(l)),\tau_x^{(1)}>l\right]\leq Cv_1(x)l^{-(k-2)/2},\quad l\geq1,
\end{align}
we  obtain
\begin{align*}
\sum_{l=1}^{n^{1-\varepsilon}}\mathbf{P}\left(\tau_x>n,\eta^+=\eta^-=l\right)
&\leq\frac{C}{n^{\alpha-\delta+(k-2)(k-3)/4}}\sum_{l=0}^{n}\mathbf{E}\left[\tilde{v}(x+S(l)),\tau_x^{(1)}>l\right]\\
&\leq\frac{Cv_1(x)}{n^{\alpha-\delta+(k-2)(k-3)/4}}\sum_{l=1}^{n}l^{-(k-2)/2}\\
&\leq Cv_1(x)\frac{\log n}{n^{\alpha-\delta+(k-2)(k-3)/4}}\\
&=o\left(n^{-\alpha/2-(k-1)(k-2)/4}\right).
\end{align*}
Furthermore, applying (\ref{L2.4}) once again, we get
\begin{align*}
&\mathbf{P}\left(\tau_x>n,\eta^+=l,\eta^-=j\right)\\
&\quad\leq C\int_W\mathbf{P}(x+S(j-1)\in dy,\tau_x>j-1,\eta^+=l)n^{-\alpha/2+\delta/2}
\frac{\tilde{v}(y)}{n^{(k-2)(k-3)/4}}\\
&\quad\leq C\int_W\mathbf{P}(x+S(l-1)\in dy,\tau_x>l-1)n^{-\alpha+\delta}\frac{v_1(y)}{n^{(k-2)(k-3)/4}}\frac{1}{(j-l)^{(k-2)/2}}.
\end{align*}
This implies that
\begin{align*}
&\sum_{l=1}^{n^{1-\varepsilon}}\sum_{j=l+1}^{n^{1-\varepsilon}}\mathbf{P}\left(\tau_x>n,\eta^+=l,\eta^-=j\right)\\
&\quad\leq\frac{C\log n}{n^{\alpha-\delta+(k-2)(k-3)/4}}\sum_{l=0}^\infty\mathbf{E}\left[v_1(x+S(l)),\tau_x>l\right]\\
&\quad=o\left(n^{-\alpha/2-(k-1)(k-2)/4}\right).
\end{align*}
As a result we have the bound
\begin{align}
\label{L2.5}
\mathbf{P}\left(\tau_x>n,\eta^+\leq\eta^-\leq\nu_n\leq n^{1-\varepsilon}\right)=
o\left(n^{-\alpha/2-(k-1)(k-2)/4}\right).
\end{align}
Combining (\ref{L2.1}) -- (\ref{L2.5}), we arrive at the inequality
\begin{align*}
\mathbf{P}\left(\tau_x>n,A_n(a)\right)
\leq n^{-\gamma k(k-1)/4}\mathbf{P}\left(\tau_x>n^{1-\gamma},A_{n^{1-\gamma}}(a)\right)
+\frac{Ca^{k-1-\alpha}}{n^{\alpha/2+(k-1)(k-2)/4}}.
\end{align*}
Iterating $N$ times we get
\begin{align*}
\mathbf{P}\left(\tau_x>n,A_n(a)\right)
\leq n^{-(1-(1-\gamma)^{N}) k(k-1)/4}\mathbf{P}\left(\tau_x>n^{(1-\gamma)^N},A_{n^{(1-\gamma)^N}}(a)\right)\\
+\frac{Ca^{k-1-\alpha}}{n^{\alpha/2+(k-1)(k-2)/4}}.
\end{align*}
Choosing $N$ sufficiently large, we arrive at the desired inequality.
\end{proof}
\begin{lemma}
\label{L3}
If $S$ is as in Theorem 1 of \cite{DW10}, then there exists a constant $C$ such that
$$
\mathbf{P}(\tau_x>n)\leq \frac{CV(x)}{n^{k(k-1)/4}},\quad x\in W.
$$
\end{lemma}
\begin{proof}
It follows from Proposition 4 of \cite{DW10} that $V(x)\sim\Delta(x)$ uniformly in $x\in W_{n,\varepsilon}$.
This and inequality (33) from \cite{DW10} imply that
$$
\mathbf{P}(\tau_x>n,\nu_n\leq n^{1-\varepsilon})\leq
\frac{C}{n^{k(k-1)/4}}\mathbf{E}\left[V(x+S(\nu_n)),\tau_x>\nu_n,\nu_n\leq n^{1-\varepsilon}\right].
$$
Recalling that the sequence $V(x+S(n)){\rm 1}\{\tau_x>n\}$ is a martingale, we conclude that
\begin{align*}
&\mathbf{E}\left[V(x+S(\nu_n)),\tau_x>\nu_n,\nu_n\leq n^{1-\varepsilon}\right]\\
&\qquad\leq
\mathbf{E}\left[V(x+S(\nu_n\wedge n^{1-\varepsilon})),\tau_x>\nu_n\wedge n^{1-\varepsilon}\right]=V(x).
\end{align*}
To complete the proof it remains to note that $\mathbf{P}(\nu_n>n^{1-\varepsilon})\leq e^{-Cn^\varepsilon}$
and that $\inf_{x\in W}V(x)>0$.
\end{proof}

\begin{lemma}\label{L4}
If $x_k=r\sqrt{n}$ and $x_1,\ldots,x_{k-1}$ are fixed, then there exists a function $\psi$ such that
\begin{equation}
\label{L4.1}
\mathbf{P}(\tau_x>n)\sim \psi(r)\frac{v_1(x)}{n^{(k-1)(k-2)/4}}.
\end{equation}
Moreover, 
\begin{equation}
\label{L4.2}
\psi(a)\leq Ca^{k-1},\quad a>0.
\end{equation}
\end{lemma}
\begin{proof}
It is clear that
$$
\mathbf{P}(\tau_x>n)=\mathbf{P}(\tau_x>n,\nu_n\leq n^{1-\varepsilon})+O\left(e^{-Cn^{\varepsilon}}\right).
$$
Furthermore,
\begin{align*}
\mathbf{P}\left(\tau_x>n,\nu_n\leq n^{1-\varepsilon},|S_k(\nu_n)|\geq \theta_n\sqrt{n}\right)
&\leq \mathbf{P}\left(\max_{j\leq n^{1-\varepsilon}}|S_k(j)|\geq \theta_n\sqrt{n}\right)\mathbf{P}(\tau_x^{(1)}>n)\\
&=o\left(n^{-(k-1)(k-2)/4}\right)
\end{align*}
and, in view of Lemma 16 from \cite{DW10},
\begin{align*}
\mathbf{P}(\tau_x^{(1)}>n,|S(\nu_n)|>\sqrt{n},\nu_n\leq n^{1-\varepsilon})=o\left(n^{-(k-1)(k-2)/4}\right).
\end{align*}
As a result we have
\begin{equation}
\label{L4.3}
\mathbf{P}(\tau_x>n)=\mathbf{P}(\tau_x>n,|S(\nu_n)|\leq\theta_n\sqrt{n},\nu_n\leq n^{1-\varepsilon})+o\left(n^{-(k-1)(k-2)/4}\right).
\end{equation}

Applying inequality (33) from \cite{DW10}, we obtain the bound
\begin{align*}
&\mathbf{P}\left(\tau_x>n,\nu_n\leq n^{1-\varepsilon},|S(\nu_n)|\leq \theta_n\sqrt{n}\right)\\
&\quad\leq\frac{C}{n^{k(k-1)/4}}\mathbf{E}\left[\Delta(x+S(\nu_n)),\tau_x>\nu_n,|S(\nu_n)|\leq \theta_n\sqrt{n}\right]\\
&\quad\leq\frac{Cr^{k-1}}{n^{(k-1)(k-2)/4}}\mathbf{E}\left[\Delta^{(k-1)}(x+S(\nu_n)),\tau_x^{(1)}>\nu_n\right]
\end{align*} 
Noting that the expectation on the right converges to $v_1(x)$ and taking into account (\ref{L4.3}), we obtain finally
\begin{equation}
\label{L4.4}
n^{(k-1)(k-2)/4}\mathbf{P}(\tau_x>n)\leq Cr^{k-1}v_1(x).
\end{equation}

Using coupling one can show that, uniformly in $x=(x_1,x_2,\ldots,x_k)\in W_{n,\varepsilon}$ with $|x_j|\leq\theta_n\sqrt{n}$
and $|x_k-r\sqrt{n}|\leq \theta_n\sqrt{n}$, holds
\begin{align*}
\mathbf{P}(\tau_x>n)\sim \mathbf{P}(\tau_x^{bm}>n)\sim \frac{\Delta^{(k-1)}(x)}{n^{(k-1)(k-2)/4}}\psi(r),
\end{align*}
where
\begin{equation}\label{L4.5}
\psi(r):=\lim_{a\to0}
\frac{\mathbf{P}\left(B_1(t)<a+B_2(t)<\ldots<(k-2)a+B_{k-1}(t)<r+B_k(t),\ t\leq1\right)}
{\mathbf{P}\left(B_1(t)<a+B_2(t)<\ldots<(k-2)a+B_{k-1}(t),\ t\leq1\right)}.
\end{equation}
Consequently,
\begin{align*}
&\mathbf{P}\left(\tau_x>n,\nu_n\leq n^{1-\varepsilon},|S(\nu_n)|\leq \theta_n\sqrt{n}\right)\\
&\sim\quad\frac{\psi(r)}{n^{(k-1)(k-2)/4}}
\mathbf{E}\left[\Delta^{(k-1)}(x+S(\nu_n))\tau_x>\nu_n,\nu_n\leq n^{1-\varepsilon},|S(\nu_n)|\leq \theta_n\sqrt{n}\right]\\
&\sim \quad\frac{\psi(r)}{n^{(k-1)(k-2)/4}}v_1(x),
\end{align*}
where in the last step we have used Lemmas 15 and 16 from \cite{DW10}. Combining this relation with (\ref{L4.3}), we get (\ref{L4.1}),
and (\ref{L4.2}) follows from (\ref{L4.4}).
\end{proof}

\begin{proof}[Proof of Theorem \ref{T1}]
Denote
$$
T^+=\min\{j\geq1:X_k(j)\geq an^{1/2}\},\quad T^-=\min\{j\geq1:X_1(j)\leq -an^{1/2}\}
$$
and
$$
T=\min\{T^+,T^-\}.
$$

We first derive an upper bound for $\mathbf{P}(\tau_x>n)$. Our starting point will be the following inequality
\begin{equation}
\label{T1.2}
\mathbf{P}(\tau_x>n)\leq\sum_{l=1}^{n/2}\mathbf{P}(\tau_x>n,T=l)+\mathbf{P}(\tau_x>n/2,T>n/2).
\end{equation}
According to Lemma \ref{L2},
\begin{equation}\label{T1.3}
\mathbf{P}(\tau_x>n/2,T>n/2)\leq \frac{C a^{k-1-\alpha}}{n^{\alpha/2+(k-1)(k-2)/4}}.
\end{equation}
Applying Lemma \ref{L3} to $(S_1,S_2,\ldots,S_{k-1})$, we conclude that, for every $l\leq n/2$, holds
\begin{align*}
\mathbf{P}(\tau_x>n,T^+=l)&\leq\int_W\mathbf{P}(x+S(l)\in dy,\tau_x>l,T^+=l)\mathbf{P}(\tau_y^{(1)}>n/2)\\
&\leq\frac{C}{n^{(k-1)(k-2)/4}}\mathbf{E}\left[v_1(x+S(l)),\tau_x>l,T^+=l\right]\\
&\leq \frac{C}{n^{(k-1)(k-2)/4}}\frac{p}{(an^{1/2})^\alpha}\mathbf{E}\left[v_1(x+S(l-1)),\tau_x>l-1\right].
\end{align*}
And an analogous inequality holds for $\mathbf{P}(\tau_x>n,T^-=l)$. As a result we have
\begin{equation}
\label{T1.4}
\sum_{l=N}^{n/2}\mathbf{P}(\tau_x>n,T^+=l)\leq \frac{Ca^{-\alpha}}{n^{\alpha/2+(k-1)(k-2)/4}}\sum_{l=N}^{\infty}\mathbf{E}\left[v(x+S(l-1)),\tau_x>l-1\right].
\end{equation}

For every fixed $l$ we have
\begin{align*}
&\mathbf{P}(\tau_x>n,T^+=l)=\int_W\mathbf{P}(x+S(l)\in dy,\tau_x>l,T^+=l)\mathbf{P}(\tau_y>n-l)\\
&\quad\sim n^{-(k-1)(k-2)/4}\mathbf{E}\left[v_1(x+S(l))\psi\left(\frac{X_k(l)}{\sqrt{n}}\right),\tau_x>l,T^+=l\right]\\
&\quad\sim n^{-(k-1)(k-2)/4}\mathbf{E}\left[v_1(x+S(l))\psi\left(\frac{X_k(l)}{\sqrt{n}}\right),\tau_x>l,T^+=l\right]\\
&\quad\sim n^{-(k-1)(k-2)/4}\mathbf{E}\left[v_1(x+S(l-1)),\tau_x>l-1\right]
\mathbf{E}\left[\psi\left(\frac{X_k(l)}{\sqrt{n}}\right),T^+=l\right].
\end{align*}
Noting that
$$
\mathbf{E}\left[\psi\left(\frac{X_k(l)}{\sqrt{n}}\right),T^+=l\right]
\sim pn^{-\alpha/2}\int_a^\infty\psi(y)\alpha y^{-\alpha-1}dy=:\theta(a),
$$
we obtain
\begin{align*}
\mathbf{P}(\tau_x>n,T^+=l)\sim p\theta(a)n^{-\alpha/2-(k-1)(k-2)/4}\mathbf{E}\left[v_1(x+S(l-1))\tau_x>l-1\right].
\end{align*}
In the same way one can get
\begin{align*}
\mathbf{P}(\tau_x>n,T^-=l)\sim q\theta(a)n^{-\alpha/2-(k-1)(k-2)/4}\mathbf{E}\left[v_2(x+S(l-1))\tau_x>l-1\right].
\end{align*}
Therefore,
\begin{equation}
\label{T1.5}
\sum_{l=1}^{N-1}\mathbf{P}(\tau_x>n,T=l)\sim \theta(a)n^{-\alpha/2-(k-1)(k-2)/4}
\sum_{l=1}^{N-1}\mathbf{E}\left[v(x+S(l-1))\tau_x>l-1\right].
\end{equation}
Combining (\ref{T1.3}) --- (\ref{T1.5}) and noting that (\ref{L4.2}) yields $\theta(a)\leq\theta(0)<\infty$, we see that
\begin{align*}
\limsup_{n\to\infty}n^{\alpha/2+(k-1)(k-2)/4}\mathbf{P}(\tau_x>n)\leq
\theta(0)\sum_{l=1}^{\infty}\mathbf{E}\left[v(x+S(l-1))\tau_x>l-1\right]\\
+Ca^{-\alpha}\sum_{l=N}^{\infty}\mathbf{E}\left[v(x+S(l-1)),\tau_x>l-1\right]+
Ca^{k-1-\alpha}.
\end{align*}
Letting here first $N\to\infty$ and then $a\to0$, we get
\begin{equation}
\label{T1.6}
\limsup_{n\to\infty}n^{\alpha/2+(k-1)(k-2)/4}\mathbf{P}(\tau_x>n)\leq \theta(0)U(x).
\end{equation}
To obtain a corresponding lower bound we note that, for every $N\geq1$,
$$
\mathbf{P}(\tau_x>n)\geq\sum_{l=1}^{N-1}\mathbf{P}(\tau_x>n,T=l).
$$
Using now (\ref{T1.5}), we have
$$
\liminf_{n\to\infty}n^{\alpha/2+(k-1)(k-2)/4}\mathbf{P}(\tau_x>n)\geq
\theta(a)\sum_{l=1}^{N-1}\mathbf{E}\left[v(x+S(l-1))\tau_x>l-1\right].
$$
Since $N$ can be chosen arbitrary large
$$
\liminf_{n\to\infty}n^{\alpha/2+(k-1)(k-2)/4}\mathbf{P}(\tau_x>n)\geq
\theta(a)U(x).
$$
Finally, it follows from (\ref{L4.2}) that $\theta(a)=\theta(0)+O(a^{k-1-\alpha})$.
Hence,
$$
\liminf_{n\to\infty}n^{\alpha/2+(k-1)(k-2)/4}\mathbf{P}(\tau_x>n)\geq
\theta(0)U(x).
$$
>From this inequality and (\ref{T1.6}) we conclude that (\ref{T1.1}) holds with $\theta=\theta(0)$.
\end{proof}
\section{Proof of Theorem~\ref{T2}}
We have to show that
\begin{equation}
\label{T2.1}
\mathbf{E}[f(X^{(n)})|\tau_x>n]\to\mathbf{E}[f(X)]
\end{equation}
for every bounded and continuous $f:C[0,1]\to\mathbb{R}$.

We first note that it suffices to prove that
\begin{align}
\label{T2.2}
\nonumber
&\lim_{a\to0}\lim_{n\to\infty}n^{\alpha/2+(k-1)(k-2)/4}\mathbf{E}[f(X^{(n)}), T=l,X_k(l)>an^{1/2},\tau_x>n]\\
&\hspace{1cm}=p\mathbf{E}[v_1(x+S(l-1)),\tau_x>l-1]\mathbf{E}[f(X),X_k(0)>0].
\end{align}
for every fixed $l$. Indeed, in view of the symmetry,
\begin{align}
\label{T2.3}
\nonumber
&\lim_{a\to0}\lim_{n\to\infty}n^{\alpha/2+(k-1)(k-2)/4}\mathbf{E}[f(X^{(n)}), T=l,X_1(l)<-an^{1/2},\tau_x>n]\\
&\hspace{1cm}=q\mathbf{E}[v_2(x+S(l-1)),\tau_x>l-1]\mathbf{E}[f(X),X_1(0)<0].
\end{align}
Then, combining (\ref{T2.2}) and (\ref{T2.3}), we get
\begin{align*}
\nonumber
&\lim_{a\to0}\lim_{n\to\infty}\mathbf{E}[f(X^{(n)}), T\leq N|\tau_x>n]\\
&\hspace{1cm}\sim\frac{\sum_{l=0}^{N-1}\mathbf{E}[v(x+S(l-1)),\tau_x>l-1]}{U(x)}\mathbf{E}[f(X)]
\end{align*}
Using Proposition~\ref{P1} and Lemma~\ref{L2}, we get (\ref{T2.1}).

In order to prove (\ref{T2.2}) we assume first that our random walk starts from $x$ with $|x_1|<A,\ldots,|x_{k-1}|<A$
and $x_k>an^{1/2}$. It is easy to see that
\begin{align*}
&\mathbf{P}\left(\tau_x>n,\nu_n\leq n^{1-\varepsilon},|S_k(\nu_n)|>n^{1/2-\varepsilon/4}\right)\\
&\leq\mathbf{P}\left(\max_{j\leq n^{1-\varepsilon}}|S_k(\nu_n)|>n^{1/2-\varepsilon/4}\right)\mathbf{P}(\tau_x^{(1)}>n)
=o\left(\frac{1}{n^{(k-1)(k-2)/4}}\right).
\end{align*}
Furthermore, it follows from Lemma~16 of \cite{DW10} that
$$
\mathbf{P}\left(\tau_x^{(1)}>n,\nu_n\leq n^{1-\varepsilon},|S(\nu_n)|>\theta_n n^{1/2}\right)
=o\left(\frac{1}{n^{(k-1)(k-2)/4}}\right).
$$
As a result we have the following representation
\begin{align*}
\mathbf{E}\left[f(X^{(n)}),\tau_x>n\right]&=
\mathbf{E}\left[f(X^{(n)}),\tau_x>n,|S(\nu_n)|\leq\theta_n n^{1/2},\nu_n\leq n^{1-\varepsilon}\right]\\
&\hspace{2cm}+o\left(\frac{1}{n^{(k-1)(k-2)/4}}\right).
\end{align*}
Further,
\begin{align*}
&\mathbf{E}\left[f(X^{(n)}),\tau_x>n,|S(\nu_n)|\leq\theta_n n^{1/2},\nu_n\leq n^{1-\varepsilon}\right]\\
&=\sum_{l=1}^{n^{1-\varepsilon}}
\int_W\mathbf{P}\left(x+S_l\in dy,\tau_x>l,\nu_n=l,|S(l)|\leq\theta_n n^{1/2}\right)
\mathbf{E}\left[f_{l,y}(X^{(n)}),\tau_y>n-l\right],
\end{align*}
where 
$$
f_{l,y}(u)=f\left(y{\rm 1}_{\{t\leq l/n\}}+u(t){\rm 1}_{\{t> l/n\}}\right),\quad u\in C[0,1].
$$ 

Using coupling, we obtain
\begin{align*}
\mathbf{E}\left[f_{l,y}(X^{(n)}),\tau_y>n-l\right]
&\sim \mathbf{E}[f(X)|X_k(0)=x_k/\sqrt{n}]\mathbf{P}(\tau^{bm}_y>n)\\
&\sim \mathbf{E}[f(X)|X_k(0)=x_k/\sqrt{n}]\frac{\Delta^{(1)}(y)\psi(x_k/\sqrt{n})}{n^{(k-1)(k-2)/4}}.
\end{align*}
This implies that
\begin{align*}
&\mathbf{E}\left[f(X^{(n)}),\tau_x>n,|S(\nu_n)|\leq\theta_n n^{1/2},\nu_n\leq n^{1-\varepsilon}\right]\\
&\hspace{1cm}\sim\frac{\mathbf{E}[f(X)|X_k(0)=x_k/\sqrt{n}]\psi(x_k/\sqrt{n})}{n^{(k-1)(k-2)/4}}\\
&\hspace{3cm}\times\mathbf{E}\left[\Delta^{(1)}(x+S(\nu_n)),\tau_x>\nu_n,\nu_n\leq n^{1-\varepsilon},|S(\nu_n)|\leq\theta_n n^{1/2}\right]\\
&\hspace{1cm}\sim v_1(x)\frac{\mathbf{E}[f(X)|X_k(0)=x_k/\sqrt{n}]\psi(x_k/\sqrt{n})}{n^{(k-1)(k-2)/4}},
\end{align*}
where in the last step we used Lemmas 15 and 16 from \cite{DW10}. Therefore,
\begin{align*}
&\mathbf{E}[f(X^{(n)}), T=l,X_k(l)>an^{1/2},\tau_x>n]\\
&\hspace{1cm}\sim\int_W\mathbf{P}(x+S(l)\in dy,\tau_x>l,T=l,X_k(l)>an^{1/2})\\
&\hspace{3cm}\times\frac{v_1(y)}{n^{(k-1)(k-2)/4}}
\mathbf{E}[f(X)|X_k(0)=y_k/\sqrt{n}]\psi(y_k/\sqrt{n})\\
&\hspace{1cm}\sim \frac{p\mathbf{E}[v_1(x+S(l-1),\tau_x>l-1]}{n^{\alpha/2+(k-1)(k-2)/4}}
\int_a^\infty\mathbf{E}[f(X)|X_k(0)=z]\psi(z)\alpha z^{-\alpha-1}dz.
\end{align*}
Since
$$
\lim_{a\to0}\int_a^\infty\mathbf{E}[f(X)|X_k(0)=z]\psi(z)\alpha z^{-\alpha-1}dz=\mathbf{E}[f(X),X_k(0)>0],
$$
the previous relation implies (\ref{T2.1}). Thus, the proof is finished.

\vspace{12pt}

{\bf Acknowledgement} We are grateful to Martin Kolb for
attracting our attention to paper \cite{JW02}.

\end{document}